\documentclass{amsart}

\usepackage{amssymb}
\usepackage{amsmath}
\usepackage{amsthm}

\usepackage[english]{babel}

\usepackage[T2A]{fontenc}
\usepackage[cp1251]{inputenc}
\newtheorem {theorem} {Theorem}[section]

\newtheorem {lemma}  [theorem]{Lemma}
\numberwithin{equation}{section}
\theoremstyle{definition}
\newtheorem{df}{Definition}

\theoremstyle{remark}
\newtheorem{rem}{Remark}

\newcommand{\rbmo}{\mathrm{RBMO}}
\newcommand{\krl}{\mathcal{K}}
\newcommand{\KK}{K}
\newcommand{\dist}{d}
\newcommand{\rd}{\mathbb{R}^m}

\newcommand{\al}{\alpha}
\newcommand{\de}{\delta}

\newcommand{\er}{\varepsilon}

\newcommand{\Rbb}{\mathbb R}

\begin{document}

\title{Calder\'on--Zygmund operators on RBMO}



\author{Evgueni Doubtsov}

\address{St.~Petersburg Department
of Steklov Mathematical Institute, Fontanka 27, St.~Petersburg
191023, Russia}

\email{dubtsov@pdmi.ras.ru}

\author{Andrei V.~Vasin }

\address{St.~Petersburg Department
of Steklov Mathematical Institute, Fontanka 27, St.~Petersburg
191023, Russia}

\email{andrejvasin@gmail.com}

\thanks{This research was supported by the Russian Science Foundation (grant No.~18-11-00053).}



\begin{abstract}
Let $\mu$ be an $n$-dimensional finite positive measure on $\rd$.
We obtain a $T1$ condition sufficient for the boundedness of
Calder\'{o}n--Zygmund operators on $\textrm{RBMO}(\mu)$, the regular BMO space of Tolsa.
\end{abstract}

\maketitle

\section{Introduction}\label{s_int}

Given a positive {Radon} measure $\mu$ on $\rd$,
Tolsa \cite{T} introduced $\rbmo(\mu)$, the regular BMO space with respect to $\mu$.
This space is suitable for the non-doubling measures $\mu$ and
it has genuine properties of the classical space BMO
such as the John--Nirenberg inequality.
Moreover, it is proved in \cite{T} that a bounded on $L^2(\mu)$ Calder\'on--Zygmund operator
maps $L^\infty(\mu)$ into $\rbmo(\mu)$.
Motivated by this result, we consider Calder\'on--Zygmund operators on $\rbmo(\mu)$:
we obtain a T1 condition sufficient for the boundedness of Calder\'on--Zygmund operators
on $\rbmo(\mu)$; see Section~\ref{ss_mthm} for a precise formulation.

\subsection{Cubes and $n$-dimensional measures}
In what follows, a cube is a closed cube in $\rd$ with sides parallel to the axes and centered at
a point of $\mathrm{supp\,} \mu$. For a cube $Q$, let $\ell=\ell(Q)$ denote its side-length.
Also the notation $Q(x, \ell)$ is used to indicate explicitly the center $x$ of the cube under consideration.

As in \cite{T}, we always assume that $\mu$ is an $n$-dimensional measure on $\rd$ for a real number $n$,
$0< n \le m$.
By definition, it means that
\begin{equation}  \label{e_ndim}
\mu(Q) \le C \ell^n(Q)\ \textrm{for any cube\ } Q\subset\rd,\ \ell(Q)>0,
 \end{equation}
with a universal constant $C>0$.

\subsection{Calder\'{o}n--Zygmund operators}
Let $\dist(\cdot, \cdot)$ denote the standard distance between points of $\rd$.
A Calder\'{o}n--Zygmund kernel associated with an $n$-dimensional measure $\mu$ on $\rd$ is a measurable function
$\krl (x, y)$ on $\rd\times \rd\setminus\{(x, x): x\in \rd\}$
satisfying the following conditions:
\begin{equation}\label{e_cz1}
|\krl(x, y)|
\le C \dist^{-n}(x,y),
\end{equation}
\begin{equation}\label{e_cz3}
\aligned
 |\krl(x_1, y)- \krl(x_2, y)| + |\krl(y, x_1)
&- \krl(y, x_2)| \\
&\le C\frac{\dist^\delta(x_1, x_2)}{\dist^{n+\delta}(x_1, y)},\quad 2\dist(x_1, x_2)\le \dist(x_1, y),
\endaligned
\end{equation}
and
\begin{equation}\label{e_cz2}
\left|\int_{Q(x, R)\setminus Q(x, r)} \krl(x, y)\, d\mu(x) \right|
\le C, \quad 0<r<R,
\end{equation}
where $C>0$ is a universal constant and
$\delta$, $0 < \delta\le 1$, is a regularity constant
specific to the kernel $\krl$.

\begin{rem}
Restrictions \eqref{e_cz1} and \eqref{e_cz3} are standard.
Condition \eqref{e_cz2} is a more special cancellation property.
\end{rem}

The Calder\'{o}n-Zygmund operator associated to the kernel $\krl (x, y)$  and the measure  $\mu$ is defined as
\[Tf(x)=\int_{\rd} \krl (x, y)f(y)\, d\mu(y) \]
for $x\notin \textrm{supp}(f\mu)$.
So, in the general setting, one introduces the following truncated operators $T_\er$, $\er>0$:
\[T_\er f(x)=\int_{\rd\setminus Q(x, \er)}\krl (x, y)f(y)\, d\mu(y).
 \]
The operator $T$ is said to be bounded on $L^p(\mu)$ if the operators $T_\er$ are bounded on $L^p(\mu)$ uniformly in $\er >0$.

\subsection{Regular BMO space}
In this section, we give an equivalent definition of $\rbmo(\mu)$, the regular BMO space
introduced by Tolsa \cite{T}.

\subsubsection{Coefficients $\KK(Q, R)$}
Given two cubes $Q\subset R$ in $\rd$, put
\[
\KK(Q, R) = 1 + \sum_{j=1}^{N_{Q, R}} \frac{\mu(2^j Q)}{\ell^n(2^j Q)},
\]
where $N_{Q, R}$ is the minimal integer $k$ such that $\ell(2^k Q) \ge \ell(R)$.
Clearly, $\KK(Q, R)\ge 1$. On the other hand, $\KK(Q, R)$ is bounded
above by $C \log(\ell(R)/ \ell(Q))$ because $\mu$ is $n$-dimensional.

\subsubsection{Doubling cubes}
\begin{df}
Let $\al>1$ and $\beta > \al^n$. A cube $Q$ is called $(\al, \beta)$-doubling if
\[
\mu(\al Q) < \beta \mu(Q).
\]
\end{df}

Let $\mu$ be a Radon measure on $\rd$ and $\al>1$.
As indicated in \cite{T}, it is known that
for a sufficiently large $\beta = \beta (\al, n)$, for $\mu$-almost all $x\in \rd$ there is
a sequence of $(\alpha, \beta)$-doubling cubes $\{Q_k\}_{k=1}^\infty$ centered at $x$ and with $\ell(Q_k)$ tending
to $0$ as $k\to \infty$.
Let $\beta_n$ denote two times the infimum of the corresponding constants $\beta(4, n)$.

\begin{df}\label{d_doubling}
A cube $Q\subset\rd$ is called doubling if $Q$ is $(4, \beta_n)$-doubling.
\end{df}

\begin{rem}\label{r_alpha}
The original definition of a doubling cube
and further results in \cite{T} are given for $\alpha=2$ and under assumption $1\le n \le m$.
Nevertheless, it is known that the results of Tolsa \cite{T}
are extendable to larger values of $\alpha$ and for $0< n \le m$;
see, for example, \cite{HYY12} for a generalization
of this theory for a wide class of measures on appropriate  metric spaces.
So, in what follows, we use the above definition with
$\alpha=4$ and still refer to original results of Tolsa \cite{T}.
\end{rem}

\subsubsection{Definition of $\rbmo$}
\begin{df}\label{d_rbmo}
The space $\rbmo(\mu)$ consists of those $f\in L^1_{loc}(\mu)$ for which there exists a constant $C_{\mathfrak{E}}>0$
and a collection of constants $\{f_Q\}$ (one constant for each doubling cube $Q\subset \rd$)
such that
\begin{equation}\label{e_df_osc}
\frac{1}{\mu(Q)}
\int_Q |f - f_Q|\, d\mu
\le C_{\mathfrak{E}}
\end{equation}
and
\begin{equation}\label{e_df_K}
|f_Q - f_R|\le C_{\mathfrak{E}} \KK(Q, R)
\end{equation}
for all doubling cubes $Q$, $R$, $Q\subset R$.
Let $\|f\|= \|f\|_{\mathfrak{E}}$ denote the infimum of the corresponding constants $C_{\mathfrak{E}}>0$.
\end{df}

Standard arguments guarantee that $\|\cdot\|$ is a norm on the space $\rbmo(\mu)$ modulo constants.

\subsection{Main theorem}\label{ss_mthm}
Suppose that a Calder\'{o}n-Zygmund operator $T$ is bounded on $L^2(\mu)$.
Then, as mentioned above, $T$ maps $L^\infty(\mu)$ boundedly into $\rbmo(\mu)$.
Moreover,
in the classical situation of homogeneous metric spaces and under additional assumption $T1 =0$,
the operator $T$ is known to be bounded on
$\mathrm{BMO}$ type spaces;
see, for example, \cite[Ch.~4, Sect.~4]{KK}.
In the present paper, we obtain a $T1$ condition sufficient for the boundedness of $T$ on $\rbmo(\mu)$.

Given a cube $Q\subset \rd$, put
\[
K(Q)=K(Q,2^k Q),
\]
where $k$ is the smallest positive integer such that $\mu(2^k Q)>\frac{1}{2}\mu(\rd)$.

\begin{theorem}\label{t_main}
Let $\mu$ be a finite positive $n$-dimensional measure on $\rd$.
Let $T$ be a Calder\'on--Zygmund operator bounded on $L^2(\mu)$.
Assume that for each doubling cube $Q\subset \rd$, there exists a constant $b_Q$ such that
\begin{equation}\label{e_main_osc}
\frac{1}{\mu(Q)}
\int_Q |T1 - b_Q|\, d\mu
\le \frac{C}{\KK(Q)}\quad\textrm{for all doubling cubes}\ Q
\end{equation}
and
\begin{equation}\label{e_main_K}
|b_Q - b_R|\le C \frac{\KK(Q, R)}{\KK(Q)}\quad\textrm{for all doubling cubes}\ Q,R,\ Q\subset R,
\end{equation}
where the constant $C>0$ does not depend on $Q$ and $R$.
Then $T$ is bounded on $\rbmo(\mu)$.
\end{theorem}
\begin{rem}
We say that $T$ is bounded on
$RBMO(\mu)$ if the operators $T_\er$, $\er > 0$, are uniformly bounded on $RBMO(\mu)$. Also, \eqref{e_main_osc} and \eqref{e_main_K} similarly mean that these estimates hold for $T_\er$  uniformly in $\er > 0$.
\end{rem}
\begin{rem}
We implicitly assume in Theorem~\ref{t_main} that $T1\in L^\infty(\mu)$.
Indeed, this property follows from \eqref{e_cz2}; see Lemma~\ref{l_T1bdd}.
If $T1$ is a constant, then Theorem~\ref{t_main} clearly guarantees that
$T$ is bounded on $\rbmo(\mu)$.
\end{rem}
\begin{rem}
Property~\eqref{e_main_osc} is, in a sense, similar to the oscillation condition used in \cite{BCFST13},
where a $T1$ theorem for $\textrm{BMO}_H$ in the {H}ermite-{C}alder\'{o}n-{Z}ygmund setting is obtained.
\end{rem}

The final step in the proof of Theorem~\ref{t_main} uses $\|\cdot\|$ only; however, the following
semi-norm on $\rbmo(\mu)$ is crucial for certain auxiliary results.

\begin{df}\label{d_rbmo_A}
Let $f\in L^1_{loc}(\mu)$. Fix a constant $\rho>1$.
Let $\|f\|_{\mathfrak{A}, \rho}$ denote the infimum of the constants $C_{\mathfrak{A}}= C_{\mathfrak{A}, \rho}>0$
with the following properties: for each cube $Q$, there exists $f_Q \in \Rbb$ such that
\begin{align}
\sup_Q
\frac{1}{\mu(\rho Q)}
 \int_Q |f(x) - f_Q|\, d\mu (x)
&\le C_{\mathfrak{A}}, \label{e_A_osc}
\\
 |f_Q - f_R|
&\le C_{\mathfrak{A}} K(Q,R)\quad \text{for any two cubes\ } Q \subset R. \label{e_A_K}
\end{align}
\end{df}

\begin{rem}
The original semi-norm on $\rbmo(\mu)$ from \cite{T} is different from those introduced
in Definitions~\ref{d_rbmo} and \ref{d_rbmo_A}. Nevertheless, all these semi-norms are equivalent;
see Section~\ref{s_aux} for further details.
\end{rem}

\subsection{Notation}
As usual, the symbol $C$ denotes an absolute constant whose value can vary from line to line.
Notation $C_{\mathfrak{A}}$, $C_{\mathfrak{B}}$, etc.\ is used in certain specific situations.

\subsection{Organization of the paper}
Auxiliary results are presented in Section~\ref{s_aux}.
Section~\ref{s_cstr} is devoted to estimates related to the main technical decomposition of functions from $\rbmo$;
equivalence of Definitions~\ref{d_rbmo} and \ref{d_rbmo_A} is essential on this step.
The proof of Theorem~\ref{t_main} is given in Section~\ref{s_proof21}.

\section{Auxiliary results}\label{s_aux}

\subsection{Equivalent definitions of $\rbmo$}

As mentioned in the introduction, Definition~\ref{d_rbmo} is not the original one for $\rbmo(\mu)$ in \cite{T}.
In the present section, we show that Definition~\ref{d_rbmo}
and the definitions of the regular BMO from \cite{T} are equivalent.

Firstly, recall several notions introduced by Tolsa \cite{T}.
Let $\rho>1$ and $f \in L^1_{loc}(\mu)$.
Given a cube $Q\subset \rd$, let $\left<f \right>_Q$ denote the standard $\mu$-average of $f$ over $Q$, that is,
\[
\left<f \right>_Q = \frac{1}{\mu(Q)}\int_Q f \, d\mu.
\]

$\bullet$ Let $\|f\|_{\mathfrak{B}, \rho}$ denote the infimum of the constants
$C_{\mathfrak{B}} = C_{\mathfrak{B}, \rho}>0$
with the following properties:
\begin{align*}
\frac{1}{\mu(\rho Q)}
\int_Q |f - \langle f\rangle_{\widetilde{Q}}| d\mu
 &\le C_{\mathfrak{B}}\ \text{for any cube $Q$ (centered at some point of $\textrm{supp}(\mu)$)},
\\
|\langle f\rangle_Q - \langle f\rangle_R|
 &\le C_{\mathfrak{B}} K(Q,R)\  \text{for any two doubling cubes\ } Q \subset R,
\end{align*}
where $\widetilde{Q}$ denotes the smallest doubling cube in the sequence $Q, 4Q, 4^2Q, \dots$.

\begin{rem}
The original definition of $\widetilde{Q}$ in \cite{T} uses the sequence {$Q, 2Q, 2^2Q, \dots$}. The present definition
of $\widetilde{Q}$ is based on $\alpha=4$; related details are
given in Definition~\ref{d_doubling} and Remark~\ref{r_alpha}.
See also \cite{HYY12} for similar definitions with $\alpha>1$.
\end{rem}

$\bullet$ Let $\|f\|_{\mathfrak{C}, \rho}$ denote the infimum of the constants $C_{\mathfrak{C}}= C_{\mathfrak{C}, \rho}>0$
with the following properties:  for any cube Q
\begin{align*}
\int_Q |f - \langle f\rangle_Q|\, d\mu
&\le C_{\mathfrak{C}} \mu(\rho Q),
\\
|\langle f\rangle_Q - \langle f\rangle_R|
&\le C_{\mathfrak{C}} K(Q,R)
\left(
\frac{\mu(\rho Q)}{\mu(Q)}
+
\frac{\mu(\rho R)}
{\mu(R)}
\right)\quad\textrm{for any two cubes\ } Q \subset  R.
\end{align*}

$\bullet$ Let $\|f\|_{\mathfrak{D}}$ denote the infimum of the constants $C_{\mathfrak{D}}>0$
with the following properties:
\begin{equation}\label{e_dfD_osc}
\int_Q |f - \langle f\rangle_Q|\, d\mu
\le C_{\mathfrak{D}} \mu(Q)\quad \textrm{for any doubling cube\ } Q
\end{equation}
and
\begin{equation}\label{e_dfD_K}
 |\langle f\rangle_Q - \langle f\rangle_R|
\le C_{\mathfrak{D}} K(Q,R)\quad \textrm{for any two doubling cubes\ } Q \subset R.
\end{equation}

The property $\|f\|_{\mathfrak{B}, \rho}<\infty$ is used to define the regular BMO space in \cite{T}.
By \cite[Lemma~2.6]{T}, the norms $\|\cdot\|_{\mathfrak{A}, \rho}$ are equivalent for different $\rho>1$;
by \cite[Lemma~2.8]{T}, $\|\cdot\|_{\mathfrak{B}, \rho}$ and $\|\cdot\|_{\mathfrak{A}, \rho}$ are equivalent.
By \cite[Lemma~2.10]{T}, $\|\cdot\|_{\mathfrak{C},\rho}$ and $\|\cdot\|_{\mathfrak{D}}$ are equivalent to $\|\cdot\|_{\mathfrak{B}, \rho}$.

Standard arguments show that Definition~\ref{d_rbmo} and the property $\|f\|_{\mathfrak{D}}<\infty$ define the same space,
with equivalent norms.
Indeed, if $\|f\|_{\mathfrak{D}}<\infty$, then trivially $f\in \rbmo(\mu)$ with $C_{\mathfrak{E}} = \|f\|_{\mathfrak{D}}$.
Now, assume that $f\in \rbmo(\mu)$. Property~\eqref{e_df_osc} guarantees that
\[
|\left<f\right>_Q - f_Q|\le C_{\mathfrak{E}}
\]
for any doubling $Q$.
Hence, \eqref{e_df_osc} implies \eqref{e_dfD_osc} with $C_{\mathfrak{D}}= 2 C_{\mathfrak{E}}$;
\eqref{e_df_K} implies \eqref{e_dfD_K} with $C_{\mathfrak{D}}= 3 C_{\mathfrak{E}}$.
Therefore, $\|\cdot\|_{\mathfrak{D}}$ and $\|\cdot\|_{\mathfrak{E}}$ are equivalent.

In the arguments related to the proof of Theorem~\ref{t_main}, we will use $\|\cdot\|_{\mathfrak{A}, \rho}$.
Thus, for further reference, we separately formulate a particular conclusion from the above arguments
as the following lemma.
\begin{lemma}\label{l_AE_equiv}
Given a constant $\rho>1$, $\|\cdot\|_{\mathfrak{A}, \rho}$ is an equivalent norm on the space $\rbmo(\mu)$ modulo constants.
\end{lemma}

\subsection{John--Nirenberg inequality for $\rbmo$}
Tolsa \cite{T} proved the following version of the John--Nirenberg inequality for $\rbmo$.
\begin{theorem}[see {\cite[Theorem~3.1]{T}}]\label{t_JN}
Let $f \in\rbmo(\mu)$, $\rho> 1$, and let $\{b_Q\}_Q$ be a collection of numbers
satisfying
\begin{align*}
\sup_Q
\frac{1}{\mu(\rho Q)}
\int_Q |f(x) - b_Q|\, d\mu(x)
&\le C \|f\|_{\mathfrak{B},\rho} \\
|b_Q - b_R|
&\le
 C K(Q,R) \|f\|_{\mathfrak{B},\rho}
\end{align*}
for any two cubes $Q \subset R$,
with an absolute constant $C>0$.
Then for any cube Q and any $\lambda > 0$, we have
\[
\mu \{x \in  Q : |f(x) - b_Q| > \lambda \} \le C_{J\!N} \mu(\rho Q) \exp
\left(\frac{-c_{J\!N}\lambda}{\|f\|_{\mathfrak{B},\rho}}
\right)
\]
with $C_{J\!N}, c_{J\!N}> 0$ depending on $\rho$.
\end{theorem}
We will need the following corollary related to Definition~\ref{d_rbmo_A}.
\begin{lemma}\label{l_A_p_Q}
Let $1< p <\infty$ and $f \in\rbmo(\mu)$, $\rho>1$.
Let $\{f_Q\}_Q$ be such numbers that \eqref{e_A_osc} and \eqref{e_A_K}
hold with $C_{\mathfrak{A}} = 2\|f\|_{\mathfrak{A}, \rho}$.
Then
\[
\left( \frac{1}{\mu(\rho Q)} \int_Q |f- f_Q|^p \right)^\frac{1}{p} \le C \|f\|
\]
for any cube $Q\subset \rd$, where $C = C(C_{J\!N}, c_{J\!N}, p, \rho) > 0$.
\end{lemma}
\begin{proof}
A standard argument is applicable.
Indeed, by Theorem~\ref{t_JN} and Lemma~\ref{l_AE_equiv},
\begin{align*}
\frac{1}{\mu(\rho Q)}
\int_Q |f - f_Q|^p \, d\mu
&=
\frac{1}{\mu(\rho Q)}
\int_0^\infty
p \lambda^{p-1} \mu\{x : |f(x) - f_Q| > \lambda\}\, d\lambda \\
&\le C_{J\!N}
\int_0^\infty
p \lambda^{p-1}
\exp\left(
\frac{-c_{J\!N}\lambda}{\|f\|_{\mathfrak{A}, \rho}}\right)\, d\lambda  \\
&\le C \|f\|_{\mathfrak{A}, \rho}^p \\
&\le C \|f\|^p,
\end{align*}
as required.
\end{proof}

\subsection{Properties of $\rbmo$}

\begin{lemma}\label{l_fQ}
Let $\mu$ be a finite $n$-dimensional measure, $f\in\rbmo(\mu)$, $\rho> 1$, and let $\{f_Q\}_Q$
be numbers such that \eqref{e_A_osc} and \eqref{e_A_K} hold with $C_{\mathfrak{A}} = 2\|f\|_{\mathfrak{A},\rho} < \infty$. Then
\[
|f_Q| \le C_f \KK(Q)\quad\textrm{for all cubes}\ Q\subset \rd,
\]
where the number $C_f>0$ does not depend on $Q$.
\end{lemma}
\begin{proof}
Without loss of generality, we assume that $\mu(\rd) =1$.
Fix a cube $Q_1$ such that $\ell(Q_1)\ge 1$ and $\mu(Q_1) > 0$. Put
\[
C_1(f) = |f_{Q_1}|.
\]

Firstly, let $Q_2$ be a cube such that $\mu(Q_2)\ge\frac{1}{2}$.
Select a cube $Q_3$ such that $Q_3 \supset Q_1 \cup Q_2$.
By the choice of $f_{Q_j}$, $j=1,2,3$,
\[
\aligned
|f_{Q_2}|
&\le |f_{Q_2} - f_{Q_3}| + |f_{Q_3} - f_{Q_1}| +|f_{Q_1}| \\
&\le 2K(Q_2, Q_3)\|f\|_{\mathfrak{A},\rho} + 2K(Q_1, Q_3)\|f\|_{\mathfrak{A},\rho} + C_1(f).
\endaligned
\]
Since $\mu$ is $n$-dimensional, we have $\ell(Q_2)\ge \varkappa>0$.
Hence,
\[
K(Q_m, Q_3) \le 1+ \sum_{j\ge 1} \frac{\mu(2^j Q_m)}{\ell^n(2^j Q_m)}
\le 1+ \sum_{j\ge 1} \frac{1}{\ell^n(2^j Q_m)}
 \le C
\]
for $m=1,2$.
Therefore,
\[
|f_{Q_2}| \le C\|f\|_{\mathfrak{A},\rho} + C_1(f) := C_0(f).
\]

Now, consider a cube $Q\subset \rd$ such that $\mu(Q)<\frac{1}{2}$.
Let $k$ be the smallest positive integer such that $\mu(2^k Q)\ge \frac{1}{2}$. Since the cube $2^k Q$ has the properties of $Q_2$, we obtain
\[
\aligned
 |f_Q|
&\le |f_Q-f_{2^k Q}|+|f_{2^k Q}| \\
&\le 2 K(Q, 2^k Q)\|f\|_{\mathfrak{A},\rho} + C_0(f) \\
&\le C_f K(Q),
\endaligned
\]
as required.
\end{proof}

\subsection{Estimates of $T1$}

\begin{lemma}\label{l_T1bdd}
Let $\mu$ be a finite positive measure on $\rd$.
Let $T$ be a Calder\'on--Zygmund operator.
Then $T 1\in L^\infty(\mu)$.
\end{lemma}
\begin{proof}
Since $\mu$ is a finite measure, we have
\[
\aligned
|T_\er 1 (x)|
&\le \int_{\rd\setminus Q(x, 2)} \left|\krl(x,y)\right|\, d\mu(y) +
\left| \int_{Q(x, 2)\setminus Q(x,\er)} \krl(x,y)\, d\mu(y)\right| \\
&\le C\mu(\rd) + C \le C
\endaligned
\]
by \eqref{e_cz1} and \eqref{e_cz2}.
\end{proof}

\begin{lemma}\label{l_T1aver}
Let the assumptions of Theorem~\ref{t_main}(ii) hold.
Then
\begin{equation}\label{e_T1_osc}
\frac{1}{\mu(Q)}
\int_Q |T1 - \left<T1\right>_Q|\, d\mu
\le \frac{C}{\KK(Q)}\quad\textrm{for any doubling cube\ } Q
\end{equation}
and
\begin{equation}\label{e_t1_K}
|\left<T1\right>_Q - \left<T1\right>_R|\le C \frac{\KK(Q, R)}{\KK(R)}\quad\textrm{for any two doubling cubes\ } Q\subset R,
\end{equation}
where the constant $C>0$ does not depend on $Q$ and $R$.
\end{lemma}
\begin{proof}
Put $h = T1$.
Firstly, let $Q\subset \rd$ be a doubling cube.
The following standard arguments guarantee that \eqref{e_main_osc} implies \eqref{e_T1_osc}:
\[
\aligned
\int_Q |h - \left<h\right>_Q|\, d\mu
&\le \int_Q |h - b_Q|\, d\mu + \int_Q \left| b_Q - \frac{1}{\mu(Q)} \int_Q h\, d\mu\right|\, d\mu   \\
&\le \frac{C\mu(Q)}{\KK(Q)} + \frac{1}{\mu(Q)} \int_Q \int_Q |h- b_Q| \, d\mu\, d\mu\\
&\le \frac{C\mu(Q)}{\KK(Q)}.
\endaligned
\]
Secondly, for any doubling cube $R\supset Q$, we have
\[
|b_R - \left<h \right>_R| \le \frac{1}{\mu(R)}\int_R |h - b_R|\, d\mu
 \le \frac{C}{\KK(R)} \le \frac{C \KK(Q, R)}{\KK(R)}
\]
by \eqref{e_main_osc}. Clearly, $\KK(R)\le C\KK(Q)$ for $Q\subset R$.
Therefore,
\[
\aligned
|\left<h\right>_Q - \left<h\right>_R|
&\le
|\left<h \right>_Q - b_Q| + |b_Q - b_R| + |b_R - \left<h \right>_R| \\
&\le \frac{C \KK(Q, R)}{\KK(R)}
\endaligned
\]
by \eqref{e_main_K}. The proof of the lemma is finished.
\end{proof}

\section{Main construction}\label{s_cstr}
Let $T$ be a Calder\'{o}n-Zygmund operator bounded on $L^2(\mu)$.
Let $f\in \rbmo(\mu)$, $\rho=2$.
Using Definition~\ref{d_rbmo_A}, for each cube $Q\subset\rd$, select a number $f_{2Q}$ such that \eqref{e_df_osc} and \eqref{e_df_K}
hold with $C_{\mathfrak{A}, \rho} = 2\|f\|_{\mathfrak{A}, \rho}$ and with $2Q$ in the place of $Q$.
In particular, the assumptions of Lemma~\ref{l_A_p_Q} are satisfied.
Also, to explain further arguments and estimates,
it is worth mentioning that $2Q$ is not necessarily doubling even if $Q$ is a doubling cube.

In the present section, we give estimates related to the following functions:
\[
\begin{aligned}
 f_1 = f_{1,Q}
&= f_{2Q}, \\
 f_2 = f_{2,Q}
&= (f- f_{2Q}) \chi_{2Q}, \\
 f_3 = f_{3,Q}
&= (f- f_{2Q}) \chi_{\rd\setminus 2Q}.
\end{aligned}
\]

Observe that
\[
f= f_1 + f_2 + f_3.
\]
This decomposition ascends to \cite{H};
see also \cite{DV}.

Let $b_{2, Q}=0$ and
\[
b_{3, Q} = \frac{1}{\mu(Q)} \int_{Q} Tf_{3,Q} (y)\, d\mu(y).
\]

In the following lemma, we assume that $Q$ is a doubling cube.

\begin{lemma}\label{l_23}
There exists a constant $C>0$ such that
\[
\frac{1}{\mu(Q)} \int_Q |T f_k  - b_{k, Q}|\,d\mu \le C \|f\|, \quad k= 2,3,
\]
for any doubling cube $Q$.
\end{lemma}
\begin{proof}
Put
\[
I_k = \frac{1}{\mu(Q)} \int_Q |T f_k  - b_{k, Q}|\, d\mu, \quad k= 2,3.
\]
For $k=2$, we have
\[
I_2 \le \left(\frac{1}{\mu(Q)} \int_Q |T f_2|^2\, d\mu \right)^{\frac{1}{2}}
\]
by H\"older's inequality.
Since $T$ is bounded on $L^2(\mu)$, the definition of $f_2$
guarantees that
\[
\begin{aligned}
I_2 \le C \left(\frac{1}{\mu(Q)} \int_{2Q} |f - f_{2Q}|^2\, d\mu \right)^{\frac{1}{2}}.
\end{aligned}
\]
Using the doubling property of $Q$ and Lemma~\ref{l_A_p_Q} with $\rho=p=2$ and $2Q$ in the place of $Q$,
we obtain
\[
\aligned
I_2
&\le C \left(\frac{\mu(4Q)}{\mu(Q)}\right)^\frac{1}{2}
\left(\frac{1}{\mu(4Q)} \int_{2Q} |f - f_{2Q}|^2\, d\mu \right)^\frac{1}{2} \\
&\le C\|f\|,
\endaligned
\]
as required for $k=2$.

For $k=3$, we have
\[
\begin{aligned}
I_3
&= \frac{1}{\mu(Q)} \int_{Q} \left| \frac{1}{\mu(Q)} \int_Q (T f_3 (x) - Tf_3 (y))\, d\mu(y)\right| d\mu(x) \\
&\le
\frac{1}{\mu(Q)} \frac{1}{\mu(Q)} \int_Q \int_{Q}\int_{\rd\setminus 2Q} |\krl(x, u) - \krl(y, u)| |f(u) - f_{2Q}|
\, d\mu(u) \, d\mu(y)\, d\mu(x).
\end{aligned}
\]
For $x, y \in Q$ and $u\in \rd\setminus 2Q$, the defining property of $\krl(\cdot, \cdot)$ guarantees that
\[
|\krl(x, u) - \krl(y, u)| \le C\frac{\dist^{\delta}(x, y)}{\dist^{n+\delta}(x, u)}.
\]
Therefore,
\[
\begin{aligned}
I_3
&\le \frac{C}{\mu(Q)} \int_{Q}\int_{\rd\setminus 2Q} \frac{\ell^\delta}{\dist^{n+\delta}(x, u)} |f(u) - f_{2Q}|
\, d\mu(u) \, d\mu(x) \\
&\le \frac{C}{\mu(Q)} \int_{Q} \sum_{k=1}^\infty
\frac{\ell^\delta}{\dist^{n+\delta}(x, u)} \int_{2^{k+1}Q \setminus 2^k Q}|f(u) - f_{2Q}|
\, d\mu(u) \, d\mu(x) \\
&\le
C \sum_{k=1}^\infty
\frac{\ell^\delta}{(2^{k-1}\ell)^{n+\delta}} \int_{2^{k+1}Q \setminus 2^k Q}|f(u) - f_{2Q}|
\, d\mu(u)\\
&\le
C \sum_{k=1}^\infty \frac{\ell^\delta}{(2^{k}\ell)^{n+\delta}}
\left( \int_{2^{k+1}Q}|f(u) - f_{2^{k+1}Q}|\, d\mu(u)
+ \mu(2^{k+1}Q)|f_{2Q} - f_{2^{k+1}Q}|
\right)
\end{aligned}
\]
by the triangle inequality.
The choice of the constants $f_{2Q}$ and $f_{2^{k+1}Q}$ guarantees that
\[
I_3
\le
C \sum_{k=1}^\infty \frac{\ell^\delta}{(2^{k}\ell)^{n+\delta}}
\left( \mu(2^{k+2}Q) + \mu(2^{k+1}Q) \KK(2Q, 2^{k+1}Q)
\right) \|f\|_{\mathfrak{A},2}.
\]
Since $\mu$ is $n$-dimensional, we have
\[
\begin{aligned}
I_3
&\le C \|f\|_{\mathfrak{A},2} \sum_{k=1}^\infty \frac{\ell^\delta(2^{k+2}\ell)^{n}}{(2^{k}\ell)^{n+\delta}}
\KK(2Q, 2^{k+1}Q) \\
&\le C \|f\|_{\mathfrak{A},2} \sum_{k=1}^\infty \frac{\KK(2Q, 2^{k+1}Q)}{2^{k\delta}} \\
&\le C \|f\|_{\mathfrak{A},2}
 \left( 1+ \sum_{k=1}^\infty 2^{-k\delta} \sum_{j=2}^{k+1} \frac{\mu(2^j Q)}{(\ell 2^j)^n} \right) \\
&\le C \|f\|
\end{aligned}
\]
by Lemma~\ref{l_AE_equiv}.
The proof of the lemma is finished.
\end{proof}

In the following lemma, the cubes under consideration are not assumed to be doubling.
However, in the proof of Theorem~\ref{t_main}, we apply this lemma in the doubling setting.

\begin{lemma}\label{l_23_K}
There exists a constant $C>0$ such that
\[
|b_{k, Q} - b_{k, R}| \le C \|f\| \KK(Q, R), \quad k= 2, 3.
\]
for any two
cubes $Q\subset R$.
\end{lemma}

\begin{proof}
For $k=2$, the required estimate is trivial; so assume that $k=3$.
Given cubes $Q\subset R$, we have
\[
|b_{3,Q}-b_{3,R}|\leq |b_{3,Q}-b_{3,Q_0}| + |b_{3,Q_0}-b_{3,R}| = I + J,
\]
where $Q_0=2^kQ$ and $k$ is the minimal integer such that $2^kQ\supset R$.

Observe that
\[
\aligned
I
&=\left|\frac{1}{\mu(Q_0)} \int_{Q_0} \frac{1}{\mu(Q)} \int_{Q} Tf_{3,Q} (y)\, d\mu(y) \, d\mu(z)\right.\\
&\qquad\qquad\qquad \left.- \frac{1}{\mu(Q)} \int_{Q} \frac{1}{\mu(Q_0)} \int_{Q_0} Tf_{3, Q_0} (z)\, d\mu(z) \, d\mu(y)\right| \\
&=\left|  \frac{1}{\mu(Q)} \frac{1}{\mu(Q_0)}\int_{Q}\int_{Q_0} Tf_3 (y) - Tf_3 (z)
\, d\mu(z) \, d\mu(y)\right|.
\endaligned
\]
By the definitions of $f_{3, Q}$ and $f_{3, Q_0}$,
\[
\aligned
Tf_{3, Q} (y)
&- Tf_{3, Q_0} (z)
=\int_{\rd\setminus 2Q} K(y,\cdot)(f-f_{2Q} ) d\mu- \int_ {\rd\setminus 2 Q_0}  K(z,\cdot) (f-f_{2 Q_0}) d\mu \\
&=\int_{2 Q_0\setminus 2 Q} K(y,\cdot)(f-f_{2 Q} )\, d\mu
+ \int_ {\rd\setminus 2 Q_0}  K(y,\cdot) (f_{2 Q_0}-f_{2 Q})\, d\mu \\
&\qquad\qquad\qquad + \int_ {\rd\setminus 2 Q_0}  (K(y,\cdot)-K(z,\cdot)) (f-f_{2 Q_0})\, d\mu \\
&= D + E + F.
\endaligned
\]
We split $D$ into dyadic sums as follows:
\[
\aligned
D
&=\sum_{j=1}^{k} \int_{2^{j+1}Q\setminus 2^j Q} K(y,\cdot)(f-f_{2 Q})\, d\mu \\
&=\sum_{j=1}^{k}\int_{2^{j+1}Q\setminus 2^j Q} K(y,\cdot)(f-f_{2^{j+1} Q})\, d\mu \\
&\qquad\qquad\qquad +\sum_{j=1}^{k}(f_{2^{j+1} Q}-f_{2 Q}) \int_{2^{j+1}Q\setminus 2^j Q} K(y,\cdot)\, d\mu  \\
&= D_1+D_2.
\endaligned
\]
To estimate $D_1$, we apply \eqref{e_cz1} and obtain
\[
|K(y,x)|\le \frac{C}{\dist^n(x, y)}\le C\frac{1}{(\ell(Q) 2^{j+2})^n}\quad
\textrm{for}\ y\in Q,\ x\in 2^{j+1}Q \setminus 2^j Q.
\]
Hence,
\begin{equation}\label{e_D1}
\aligned
|D_1|
&\le C\sum_{j=1}^{k}
\frac{\mu(2^{j+2}Q)}{(\ell(Q) 2^{j+2})^n}
\frac{1}{\mu(2^{j+2}Q)} \int_{2^{j+1}Q\setminus 2^j Q} |f-f_{2^{j+1} Q}|\, d\mu \\
&\le C K(Q, Q_0) \|f\|_{\mathfrak{A},2} \\
&\le C K(Q, Q_0) \|f\|
\endaligned
\end{equation}
by Lemma~\ref{l_AE_equiv}.
Below we repeatedly use the equivalence of $\|\cdot\|_{\mathfrak{A},2}$ and $\|\cdot\|$ without
explicit reference to Lemma~\ref{l_AE_equiv}.

Next, applying summation by parts, we obtain
\begin{align*}
D_2=
(f_{2^{k+1}Q} - f_{2Q})
&\int_{2^{k+1}Q\setminus 2Q} K(y,\cdot)\, d\mu  \\
&- \sum_{j=1}^{k-1} (f_{2^{j+2} Q}-f_{2^{j+1} Q})
 \int_{2^{j+1}Q\setminus 2Q} K(y,\cdot)\, d\mu.
\end{align*}
By the cancellation property \eqref{e_cz2},
\[
\left|\int_{2^{j+1}Q\setminus 2Q} K(y,\cdot)\, d\mu \right|\le C, \quad j=1, 2,\dots, k.
\]
Thus, the choice of the numbers $f_{2^{j+1}Q}$, $j=0, 1, \dots, k$, guarantees that
 \begin{equation}\label{e_D2}
\aligned
|D_2|
&\le C K(2Q, 2Q_0) \|f\|_{\mathfrak{A},2}
+  C
\sum_{j=1}^{k-1} K(2^{j+1} Q,2^{j+2} Q) \|f\|_{\mathfrak{A},2} \\
&\le C K(Q, Q_0)\|f\|_{\mathfrak{A},2}
+ C K(2Q, 2Q_0)\|f\|_{\mathfrak{A},2}\\
&\le C K(Q, Q_0)\|f\|.
\endaligned
\end{equation}

To estimate $E$, we also use the cancellation property \eqref{e_cz2} and
we obtain
 \begin{equation}\label{e_E}
\aligned
|E|
&=\left| (f_{2 Q}-f_{2 Q_0}) \int_ {\rd\setminus 2 Q_0}  K(y,\cdot)\, d\mu \right| \\
&\le C |f_{2 Q}-f_{2 Q_0}| \\
&\le C K(Q, Q_0)\|f\|_{\mathfrak{A},2} \\
&\le C K(Q, Q_0)\|f\|.
\endaligned
\end{equation}

Now, consider $F$. We have
 \[
 |K(y,x)-K(z,x) | \le C \frac{\ell^{\delta}(Q_0)} {d^{n+\delta}(x, Q_0)}
 \]
for $x \in \rd \setminus 2 Q_0$.

Therefore,
\[
\aligned
|F|
&\le C \ell^{\delta}(Q_0) \int_{\rd \setminus 2 Q_0}   \frac{|f-f_{2 Q_0}|} {d^{n+\delta}(x, Q_0)}\, d\mu (x) \\
&\le C \sum_{j=1}^\infty \frac{\ell^{\delta}(Q_0)}{(2^{j+2} \ell(Q_0))^{n+\delta}}\int_{2^{j+1}
Q_0 \setminus 2^j Q_0} |f-f_{2 Q_0}|\, d\mu \\
&\le C \sum_{j=1}^\infty \frac{\ell^{\delta}(Q_0)}{(2^{j+2} \ell(Q_0))^{n+\delta}}\left( \int_{2^{j+1}
    Q_0 \setminus 2^j Q_0} |f-f_{ 2^{j+1} Q_0}|\, d\mu \right.\\
&\qquad\qquad\qquad\qquad\qquad\qquad
   \left. +|f_{ 2^{j+1} Q_0}-f_{2 Q_0}|\int_{2^{j+1} Q_0 \setminus 2^j Q_0}\,  d\mu\right)= F_{1}+F_{2}.
 \endaligned
 \]
Firstly,
 \begin{equation}\label{e_F11}
\aligned
F_{1}
&\leq \sum_{j=1}^\infty
\frac{\ell^{\delta}(Q_0) \mu(2^{j+2} Q_0)}{(2^{j+2}
\ell(Q_0))^{n+\delta}}\|f\|_{\mathfrak{A},2} \\
&\le C \sum_{j=1}^\infty \frac{\|f\|_{\mathfrak{A},2}}{(2^j )^{\delta}} \\
&\le C \|f\|.
\endaligned
\end{equation}
Secondly,
 \[
\aligned
F_{2}
&\leq \sum_{j=1}^\infty
\frac{\ell^{\delta}(Q_0) \mu(2^{j+2} Q_0)}{(2^{j+2}  \ell(Q_0))^{n+\delta}}
\|f\|_{\mathfrak{A},2} K(2 Q_0,2^j Q_0) \\
&\le C \sum_{j=1}^\infty \frac{\|f\|_{\mathfrak{A},2} K(2 Q_0,2^j Q_0)} {(2^j )^{\delta}} .
\endaligned
\]
Since $\mu$ is $n$-dimensional, we have $K(2 Q_0,2^j Q_0) \le C j$ with a universal constant $C>0$. Thus,
\begin{equation}\label{e_F12}
F_{2}\le C \|f\|_{\mathfrak{A},2}\sum_{j=1}^\infty\frac{j}{(2^j )^{\delta}}
\le C \|f\|.
\end{equation}

Combining (\ref{e_D1}--\ref{e_F12}) and integrating with respect to $z$ and $y$,
we obtain the required estimate for $I = |b_{3,Q}-b_{3,Q_0}|$.
Therefore, it remains to estimate $J$.

We have
\[
\aligned
J
&=|b_{3,Q_0}-b_{3,R}|  \\
&=\left|\frac{1}{\mu(R)} \int_{R} \frac{1}{\mu(Q_0)} \int_{Q_0} Tf_{3,Q_0} (y)\, d\mu(y) \, d\mu(w)\right.\\
&\qquad\qquad\qquad \left.- \frac{1}{\mu(Q_0)} \int_{Q_0} \frac{1}{\mu(R)} \int_{R} Tf_{3, R} (w)\, d\mu(w) \, d\mu(y)\right| \\
&=\left|  \frac{1}{\mu(Q_0)} \frac{1}{\mu(R)}\int_{Q_0}\int_{R} Tf_{3,Q_0} (y) - Tf_{3,R} (w)
\, d\mu(w) \, d\mu(y)\right|.
\endaligned
\]

By the definitions of $f_{3,Q_0}$ and $f_{3,R}$,
\[
\aligned
Tf_{3, Q_0} (z)
&- Tf_{3, R} (w)
\!=\int_{\rd\setminus 2Q_0} K(z,\cdot)(f-f_{2Q_0} ) d\mu -
\int_{\rd\setminus 2 R}  K(w,\cdot) (f-f_{2 R}) d\mu \\
&=-\int_{2 Q_0\setminus 2 R} K(w,\cdot)(f-f_{2 R} )\, d\mu
+ (f_{2 Q_0}-f_{2 R}) \int_ {\rd\setminus 2 Q_0}  K(w,\cdot)\, d\mu \\
&\qquad\qquad\qquad + \int_ {\rd\setminus 2 Q_0}  (K(z,\cdot)-K(w,\cdot)) (f-f_{2 Q_0})\, d\mu \\
&=\mathcal{D}+ \mathcal{E} + \mathcal{F}.
\endaligned
\]

Firstly, $\mathcal{D}$ is estimated similarly to $D$ and even simpler:
\[
\aligned
\mathcal{D}
&=-\int_{2 Q_0\setminus 2 R} K(w,\cdot)(f-f_{2 Q_0} )\, d\mu
+ (f_{2 R} -f_{2 Q_0})\int_{2 Q_0\setminus 2 R} K(w,\cdot)\, d\mu
\\
&= \mathcal{D}_1 + \mathcal{D}_2.
\endaligned
\]
Since $\ell (R)$ and $\ell (Q_0)$ are comparable and $\mu$ is $n$-dimensional, we have
\begin{equation}\label{e_calD1}
\aligned
|\mathcal{D}_1|
&\le
\int_{2Q_0\setminus 2R}\frac{C}{\dist^n(R, 2Q_0\setminus 2R)} |f-f_{2 Q_0}|\, d\mu
\\
&\le
\frac{C}{\ell^n(4 R)}
\frac{\mu(4 Q_0)}{\mu(4Q_0)} \int_{2Q_0} |f-f_{2Q_0}|\, d\mu
\\
&\le C  \frac{\mu(4Q_0)}{\ell^n(4Q_0)}\|f\|_{\mathfrak{A},2} \\
&\le C  \|f\|.
\endaligned
\end{equation}

Using again that
$\ell (R)$ and $\ell (Q_0)$ are comparable and $\mu$ is $n$-dimensional,
we obtain
 \begin{equation}\label{e_calD2}
\aligned
|\mathcal{D}_2|
&\le |f_{2Q_0}-f_{2R}|
\frac{\mu(2Q_0)}{\dist^n(R, 2Q_0\setminus 2R)}
 \\
&\le C |f_{2Q_0}-f_{2R}|
\frac{\mu(2Q_0)}{\ell^n(R)}
\\
&\le C |f_{2Q_0}-f_{2R}| \\
&\le C K(2R, 2Q_0)\|f\|_{\mathfrak{A}, 2} \\
&\le C \|f\|.
\endaligned
\end{equation}

To estimate $\mathcal{E}$, we use the cancellation property \eqref{e_cz2}
and obtain
 \begin{equation}\label{e_calE}
\aligned
|\mathcal{E}|
&=\left| (f_{2 R}-f_{2 Q_0}) \int_ {\rd\setminus 2 Q_0} K(w,\cdot)\, d\mu \right| \\
&\le C |f_{2 R}-f_{2 Q_0}| \\
&\le C K(2R, 2Q_0)\|f\|_{\mathfrak{A}, 2} \\
&\le C \|f\|.
\endaligned
\end{equation}

Now, we estimate $\mathcal{F}$. Property~\eqref{e_cz3} guarantees that
\[
\aligned
|\mathcal{F}|
&= \left|\int_ {\rd\setminus 2 Q_0} (( K(z,\cdot)-K(w,\cdot)) (f-f_{2 Q_0})\, d\mu \right|
\\
&\le C \dist^\de(z,w) \int_ {\rd\setminus 2 Q_0}
\frac{|f(x) - f_{2 Q_0}|}{\dist^{n+\delta}(x, Q_0)}\, d\mu(x).
\endaligned
\]
We have
\[
  \dist^\de(z,w) \le \ell^\de(Q_0).
\]
Hence, using dyadic decompositions and summation by parts,
we repeat the arguments applied to estimate $|F|$ and we
obtain the following analog of \eqref{e_F11}
and \eqref{e_F12}:
\begin{equation}\label{e_calF}
|\mathcal{F}|\le C\|f\|.
\end{equation}

Now, combining estimates (\ref{e_calD1}--\ref{e_calF}), we conclude that
\[
J = |b_{3,Q}-b_{3,R}|\le C K(Q,R)\|f\|
\]
for any two cubes $Q \subset R$.
The proof of Lemma~\ref{l_23_K} is finished.
\end{proof}

\section{Proof of Theorem~\ref{t_main}}\label{s_proof21}

Given an $f\in\rbmo(\mu)$, we have to prove that $Tf \in \rbmo(\mu)$.
Put $g = Tf$. In this section, for every doubling cube $Q\subset\rd$,
we find a constant $g_Q \in\Rbb$ such that
\begin{equation}\label{e_osc_todo}
\frac{1}{\mu(Q)} \int_{Q}|g - g_Q|\, d\mu \le C
\end{equation}
and
\begin{equation}\label{e_K_todo}
|g_Q - g_R| \le C\KK(Q, R)\quad\textrm{for any two doubling cubes\ }  Q\subset R,
\end{equation}
where $C= C_f >0$.

So, given an $f\in\rbmo(\mu)$ and a doubling cube $Q\subset \rd$, we apply the construction
described in Section~\ref{s_cstr} and we obtain
\[f = f_{1, Q} + f_{2, Q} + f_{3, Q} = f_1 + f_2 + f_3.
\]
Also, we have constants $b_{2, Q}$ and $b_{3, Q}$.
By Lemma~\ref{l_T1aver}, properties \eqref{e_T1_osc} and \eqref{e_t1_K} hold for the constants $\left<T1\right>_Q$ and $\left<T1\right>_R$.
Put
\[
g_Q = f_{2Q} \left<T1 \right>_Q + b_{2, Q} + b_{3, Q}.
\]

\subsection{Oscillation condition \eqref{e_osc_todo}}
We have
\[
\begin{aligned}
&\frac{1}{\mu(Q)} \int_{Q}|g - g_Q|\, d\mu \\
&\le \frac{1}{\mu(Q)} \left(\int_{Q}|Tf_1 - f_{2Q}\left<T1 \right>_Q|\, d\mu + \int_{Q}|Tf_2 - b_{2,Q}|\, d\mu +
\int_{Q}|Tf_3 - b_{3,Q}|\, d\mu\right) \\
&\le \frac{1}{\mu(Q)} \int_{Q}|Tf_1 - f_{2Q}\left<T1 \right>_Q|\, d\mu + C \|f\|
\end{aligned}
\]
by Lemma~\ref{l_23}. Recall that $T f_1 = T f_{2Q} = f_{2Q} T 1$.
Therefore,
applying \eqref{e_T1_osc} and Lemma~\ref{l_fQ}, we obtain
\[
\aligned
\frac{1}{\mu(Q)} \int_{Q}|Tf_1 - f_{2Q} \left<T1 \right>_Q|\, d\mu
&=
\frac{|f_{2Q}|}{\mu(Q)}  \int_{Q}|T1 - \left<T1 \right>_Q|\, d\mu \\
&\le
\frac{C_f \KK(2Q)}{\KK(Q)} \le C_f.
\endaligned
\]
Hence, the oscillation condition \eqref{e_osc_todo} is proved.

\subsection{$\KK$-condition \eqref{e_K_todo}}
Let $Q\subset R$ be doubling cubes.
Combining the triangle inequality and Lemma~\ref{l_23_K}, we obtain
\[
\aligned
|g_Q - g_R|
&\le |f_{2Q} \left<T1\right>_Q - f_{2R}\left<T1\right>_R|
 + |b_{3, Q} - b_{3, R}| \\
&\le |f_{2Q} \left<T1\right>_Q - f_{2R}\left<T1\right>_R| + C \|f\|\KK(Q, R).
\endaligned
\]
Next, the choice of the constants $f_{2Q}$ and $f_{2R}$ guarantees that
\[
|f_{2Q} - f_{2R}|\le C\|f\|_{\mathfrak{A}, 2} \KK(2Q, 2R) \le C\|f\|\KK(Q, R).
\]
Also, we have $T1\in L^\infty(\mu)$ by Lemma~\ref{l_T1bdd}, hence $|\left<T1\right>_Q|\le C$. Therefore,
\[
\begin{aligned}
|f_{2Q} \left<T1\right>_Q - f_{2R}\left<T1\right>_{2R}|
&\le |\left<T1\right>_Q| |f_{2Q} - f_{2R}| + |f_{2R}| \left|\left<T1\right>_Q - \left<T1\right>_R \right| \\
&\le  C\|f\|\KK(Q, R) + |f_{2R}| \left|\left<T1\right>_Q - \left<T1\right>_R \right|.
\end{aligned}
\]
Applying Lemma~\ref{l_T1aver} and Lemma~\ref{l_fQ}, we obtain
\[
|f_{2R}| \left|\left<T1\right>_Q - \left<T1\right>_R \right| \le C_f \KK(2R) \frac{\KK(Q, R)}{\KK(R)}
\le C_f \KK(Q, R).
\]
Combining the above estimates, we conclude that \eqref{e_K_todo} holds.
This ends the proof of the theorem.

\bibliographystyle{amsplain}

\end{document}